\documentclass[10pt]{article}
\usepackage[leqno]{amsmath}
\usepackage{amssymb}
\usepackage{amsfonts} 
\usepackage{amstext} 
\usepackage{amsthm} 
\usepackage{euscript} 
\usepackage{setspace}
\makeatletter
\renewcommand{\theenumi}{\roman{enumi}}

\renewcommand{\p@enumii}{\theenumi--}
\makeatother

\newcommand{\Z}{\ensuremath{\mathbb{Z}}}

\newcommand{\R}{\ensuremath{\mathbb{R}}} 
\newcommand{\C}{\ensuremath{\mathbb{C}}}

\newcommand{\cE}{\ensuremath{\mathcal{E}}}

\newcommand{\cH}{\ensuremath{\mathcal{H}}}

\newcommand{\cJ}{\ensuremath{\mathcal{J}}}

\newcommand{\cL}{\ensuremath{\mathcal{L}}}

\newcommand{\cR}{\ensuremath{\mathcal{R}}} 
\newcommand{\cS}{\ensuremath{\mathcal{S}}}

\newcommand{\cX}{\ensuremath{\mathcal{X}}}

\newcommand{\w}{\ensuremath{\omega}} 
\newcommand{\wc}{\ensuremath{\w_{\text{can}}}}

\newcommand{\s}{\ensuremath{\sigma}}

\newcommand{\Diff}{\ensuremath{\text{Diff}}}

\begin{document} 


\fontsize{10}{14}

\begin{center} 
{\Large \bf Monodromy groups of Lagrangian tori in $\R^4$} \\ 

\vspace{.3in} 
Mei-Lin Yau\footnote{Research 
Supported in part by National Science Council grant  
97-2115-M-008-009-. 
 

}  
\end{center} 

\vspace{.2in} 
\begin{abstract} 
We determine the Lagrangian monodromy group $\cL(T)$ and the smooth monodromy group 
$\cS(T)$ of a Clifford torus $T$ in $\R^4$. We  show that $\cL(T)$ is isomorphic to the 
infinite dihedral group, and $\cS(T)$ is  generated by three reflections. 
We give explicit formulas for both groups. We also show that if a Lagrangian torus is smoothly isotopic 
to a Clifford torus then the smooth isotopy can be chosen to be Lagrangian outside of a disc. 
\end{abstract}

\newtheorem{theo}{Theorem}[section]
\newtheorem{cond}[theo]{Condition}
\newtheorem{defn}[theo]{Definition}
\newtheorem{exam}[theo]{Example}
\newtheorem{lem}[theo]{Lemma}
\newtheorem{cor}[theo]{Corollary}

\newtheorem{prop}[theo]{Proposition}
\newtheorem{rem}[theo]{Remark}
\newtheorem{notn}[theo]{Notation}
\newtheorem{fact}[theo]{Fact}

%
%

\section{Introduction}

In this note we work in the standard symplectic 4-space $(\R^4, \w=\sum_{j=1}^2dx_j\wedge dy_j)$ 
unless otherwise mentioned. 
Let $L\overset{\iota}{\hookrightarrow}(\R^4,\w)$ be an embedded Lagrangian torus with respect to 
the standard symplectic 2-form $\w$. The Lagrangian condition means that the pull-back 2-form $\iota^*\w=0\in \Omega^2(L)$ vanishes on $L$. Gromov \cite{G} proved that $L$ is not {\em exact}, i.e., the pull-back 1-form $\iota^*\lambda$ of a primitive $\lambda$ of $\w=d\lambda$ represents a nontrivial class in the cohomology group $H^1(L,\R)$. 

Let $\Diff_0^c(\R^4)$ denote the group of orientation preserving diffeomorphisms with compact support on $\R^4$ that are isotopic to the identity map. We are interested in studying various types of self-isotopies of $L$. 
It is well-known that to a smooth isotopy $L_s$, $s\in [0,1]$, between two embedded tori $L_0,L_1$, we may associate a  family of maps $\phi_s\in \Diff_0^c(\R^4)$ with $\phi_0=id$ such that $\phi_s(L)=L_s$. 
We will make no distinction between $L_s$ and the associated maps $\phi_s$ from now on.

A path $\phi_s\in \Diff_0^c(\R^4)$ with $0\leq s\leq 1$ and $\phi_0=id$ associates to a fixed torus $L$ a family of tori $L_s:\phi_s(L)$ in 
$\R^4$. The family of maps  $\phi_t\in \Diff_0^c(\R^4)$ 
is called a {\em smooth self-isotopy} of 
$L$ if $\phi_1(L)=L$. Moreover, if all $L_s$ are Lagrangian with respect to $\w$ ($\w$-Lagrangian) then $\phi_s$ is called a {\em Lagrangian self-isotopy} of $L$. This is equivalent to say that 
$L$ is $\phi_s^*\w$-Lagrangian. Suppose in addition that the cohomology class of $\iota^*\phi^*_s\lambda$ is independent of $s$, then $\phi_s$ is called a {\em Hamiltonian self-isotopy} of $L$. 
Equivalently, $\phi_s$ is Hamiltonian if it is generated by a Hamiltonian vector field.  
Each self-isotopy $\phi_s$ of $L$ associates an isomorphism 
\[ 
(\phi_1)_*:H_1(L,\Z)\to H_1(L,\Z), 
\] 
called 
a {\em smooth} (resp. {\em Lagrangian, Hamiltonian}) monodromy of $L$ if $\phi_t$ is smooth (resp. Lagrangian, Hamiltonian). The group of all smooth monodromies of $L$ is called the {\em smooth monodromy group} (SMG)  of $L$, and is denoted by $\cS(L)$. Likewise, $\cL(L)$ and $\cH(L)$ denote the {\em Lagrangian monodromy group} (LMG) and the {\em Hamiltonian monodromy group} (HMG) of 
$L$ respectively. It is easy to see that $\cH(L)\subset \cL(L)\subset \cS(L)$. Note that though we focus only on Lagrangian 
$2$-tori here, the groups $\cH(L)$, $\cL(L)$ and 
$\cS(L)$ are defined for any embedded Lagrangian submanifold 
$L$ of any dimension. 

The interest in such monodromy groups is  to study the Lagrangian knot problem \cite{EP3} 
from a different perspective. Clearly if $L$ and $L'$ are smoothly isotopic, then their smooth monodromy 
groups are isomorphic. Similar conclusion holds for the  Lagrangian and the Hamiltonian cases as well.
In \cite{Y1} we studied $\cH(L)$ for $L$ being either a monotone Clifford torus or a Chekanov torus. The latter was constructed (and 
called a {\em special torus}) by Chekanov in \cite{C}. 
 We proved that these two 
tori are distinguished by their spectrums associated to their 
Hamiltonian monodromy groups \cite{Y1}. Another result 
concerning $\cH(L)$ was obtained by Hu, Lalonde and Leclercq  in their preprint \cite{HLaLe}, in which they  proved that 
the Hamiltonian monodromy group $\cH(L)$  is trivial for any weakly exact Lagrangian submanifold $L$ of a symplectic manifold. 
In this note we will focus on $\cL(L)$ and $\cS(L)$ instead. 

Recall from \cite{P} that the Maslov class $\mu=\mu_L\in  H^1(L,\Z)$ of a Lagrangian torus $L\subset \R^4$ is nonzero with divisibility 2. An element $h\in \cL(L)$ clearly has to satisfy $\mu\circ h=\mu$. 
 Note that in general symplectic manifolds, 
$h\in \cL(L)$ also has to preserve the {\em linking class} 
$\ell_L\in H^1(L,\Z)$ (see \cite{EP2} and Section \ref{maslink})  
whenever defined. However, since $\ell_L=0$ for any embedded 
$L\subset \R^4$ \cite{EP2}, it imposes no further restriction on 
$\cL(L)$.  
Let $G_\mu$ denote the formal subgroup of all group isomorphisms 
$g:H_1(L,\Z)\to H_1(L,\Z)$ such that $\mu\circ g=\mu$. Clearly $\cL(L)$ is a subgroup of $G_\mu$. Our 
first result is the following: 

\begin{theo} \label{cL} 
Assume that $T$ is a Clifford torus. Then $\cL(T)=G_\mu$. 
\end{theo} 
The group $G_\mu$ is freely generated by two generalized 
reflections $f_0,f_1$ (see (\ref{g+}), (\ref{g-}), (\ref{gmu}) in Section \ref{lmg}) with $f_i(\gamma_0)=-\gamma_0$, where 
$\gamma_0\in H_1(T,\Z)$ is a primitive class with $\mu_T(\gamma_0)=0$, hence $G_\mu$ is isomorphic to the infinite dihedral group 
$D_\infty$ \cite{Hum}.

For the smooth counterpart, we obtain the following result 
due to the vanishing of $\ell_L$: 

\begin{theo} \label{L0L1} 
Let $L_s=\phi_s(L_0)$, $0\leq s\leq 1$, $\phi_0=id$, be a smooth isotopy between two Lagrangian tori $L_0,L_1\subset \R^4$. Then 
for any $\gamma\in H_1(L_0,\Z)$, 
\[ 
\mu({\phi_1}_*(\gamma))-\mu(\gamma)\in 4\Z . 
\] 
I.e., 
\[ 
\phi_1^*\mu-\mu \in H^1(L_0,\Z) \quad \text{ has divisibility 4}. 
\] 
\end{theo} 
Thus $\cS(L)$ is a subgroup of 
\[ 
\cX=\cX_L:=\{ g\in \text{Isom}(H_1(L,\Z))\mid \mu_L\circ g-\mu_L\in 4\cdot H^1(L,\Z)\}. 
\] 

We determine $\cS(L)$ for the case of a Clifford torus: 

\begin{theo} \label{cS} 
Assume that $T$ is a Clifford torus, then $\cS(T)=\cX_T$.   
In particular, $\cS(T)$ is generated by $\cL(T)$ and a reflection along 
a class $\gamma\in H_1(T,\Z)$ with $\mu_T(\gamma)=2$. 
\end{theo} 

It turns out that any smooth isotopy 
between a Lagrangian torus and a Clifford torus can be modified at 
either end by a self-isotopy to match the Maslov classes at both ends. 
We have the following:

\begin{prop}  \label{StoL} 
Let $L\subset\R^4$ be an embedded Lagrangian torus smoothly isotopic to a Clifford torus $T$. Then there exists a smooth isotopy 
$\phi_s\in\Diff^C_0(\R^4)$, $s\in [0,1]$  with $\phi_0=id$, 
$\phi_1(T)=L$, and $\phi_1$ preserves the corresponding 
Maslov classes. i.e., 
\[ 
\phi^*_1\mu_L=\mu_T. 
\] 
Moreover, one can modify $\phi_s$ so that 
$\phi_s(T\setminus D)$ is Lagrangian for $s\in[0,1]$, 
where $D\subset T$ is an embedded disc. 
\end{prop}

However, at 
the present stage we do not know how to improve $\phi_s(T)$ to a genuine Lagrangian isotopy between $T$ and $L$. To achieve the goal, it seems necessary 
(and perhaps enough) to have a better understanding of  the isotopy of Lagrangian discs with prescribed boundary conditions. 

We remark here that in \cite{Mo} K. Mohnke showed that all  
embedded Lagrangian tori in $\R^4$ are smoothly isotopic to 
a Clifford torus. Also in his thesis \cite{I} A. Ivrii showed that any 
embedded Lagrangian torus in $\R^4$ is Lagrangian isotopic 
to a Clifford torus. Both of them used pseudoholomorphic curve 
techniques \cite{G} and methods of symplectic field
theory \cite{EGH, BEHWZ}. 

This article is organized as follows. In Section \ref{maslink} we  
review necessary background on Maslov class and the linking 
class. In Section \ref{loop} we discuss the framings of the symplectic normal bundle of a loop in $\R^4$, and the change of 
framings under diffeomorphisms. Theorem \ref{cL} is proved in 
Section \ref{lmg}. Theorem \ref{L0L1} is proved in the beginning 
of Section \ref{smg}, followed by the proof of Theorem \ref{cS} 
which consists of Propositions \ref{mcg}-\ref{cX}. 
Proposition \ref{StoL} is proved in Section \ref{stol}. 
We will use the convention $S^1\cong \R/2\pi \Z$ throughout the paper. 

\section{Maslov class and linking class} \label{maslink}

As we are concerned with monodromies of self-isotopies of a Lagrangian 
torus, we should discuss at first two relevant  classes in $H^1(L,\Z)$:  the {\em Maslov class} $\mu=\mu_L$ (see \cite{MS} for more details) and the 
{\em linking class} $\ell=\ell_L$. The latter is defined (and denoted by $\sigma$) in \cite{EP2}. 

\vspace{.1in} 
\noindent 
{\bf Maslov class.} \  
The Maslov class $\mu$ is defined as follows. Given $\gamma\in H_1(L,\Z)$ 
and let $C\subset L$ be an immersed curve representing $\gamma$, the tangent bundle $T_CL$ over $C$ is a closed path of Lagrangian planes and hence 
a cycle in the Grassmannian of Lagrangian planes in the symplectic vector space $\R^4$. Then $\mu(\gamma)$ is 
defined to be the {\em Maslov index} of the cycle $T_CL$. 

\begin{theo}[\cite{P}] 
The Maslov class $\mu$ of a Lagrangian torus $L\subset \R^4$ is nontrivial and is of divisibility two. 
\end{theo} 

\begin{exam} 
{\rm 
Consider   a {\em Clifford torus}  
\[ 
T=T_{a,b}:=\{ (ae^{ i t_1},be^{i t_2})\in \C^2\mid t_1,t_2\in S^1\cong\R/2\pi \Z\}.  
\] 
Let $\gamma_1\in H_1(T,\Z)$ be the class represented by the curve 
$\{ (ae^{ i t_1},b)\in \C^2\mid t_1\in \R/2\pi\Z\}$, and 
$\gamma_2\in H_1(T,\Z)$  the class represented by 
$\{ (a,be^{i t_2})\in \C^2\mid t_2\in \R/2\pi\Z\}$. 
Then $\mu_T(\gamma_1)=2=\mu_T(\gamma_2)$. 
} 
\end{exam} 

That $\mu_L\neq 0$ implies that the Lagrangian monodromy group $\cL(L)$ can only be a proper subgroup of 
$\text{Isom}(H_1(L,\Z))\cong GL(2,\Z)$. 

\vspace{.1in} 
\noindent 
{\bf Linking class.} \  
The linking class $\ell=\ell_L\in H^1(L,\Z)$ is defined as follows. Take $v$ to be any 
non-vanishing vector field on $L$ which is {\em homotopically trivial}, i.e., $v$ is  homotopic to some $v'$ 
in the space of non-vanishing vector fields on $L$, such that $v'$  generates the  kernel of a non-vanishing closed 1-form on $L$. 
 Let $J$ be an $\w$-compatible almost complex structure on $\R^4$.  Then 
$\ell (\gamma):=lk (C+\epsilon Jv, L)$ is defined to be the linking number with $L$ of 
the push-off of $C$ in the direction of $Jv$, where $C\subset L$ is an immersed curve representing  the class $\gamma$. 

The class $\ell$ is independent of the choices involved. That 
$\ell(\gamma)$ is independent of $J$ can be seen as follows. 
First of all, the space of $\w$-compatible almost complex 
structures is contractible and, since $L$ is Lagrangian, 
$Jv$ is transversal to $L$ for any $\w$-compatible $J$. 
So in particular we can take $J$ to be $J_0$ the standard 
complex structure on $\R^4$. Secondly, 
the independence of $v$ 
comes from the observation that 
vector fields generating the kernels of non-vanishing closed 
1-forms on $L$ are homotopic as nowhere vanishing vector 
fields. Finally, if $C,C'$ are two representatives of $\gamma$ 
then, since $H_1(L)$ is abelian, $C$ and $C'$ are free 
homotopic. Hence $\ell(\gamma)$ is independent of the 
choices of $v,J$ and $C$ with the prescribed conditions.

\begin{exam} 
{\rm Let $C\subset L$ be an embedded closed curve representing a nontrivial class $\gamma\in H_1(L,\Z)$. Parameterize $C$ by 
$t\in  S^1\cong \R/2\pi\Z$ so 
that its tangent vector field $\dot{C}(t)$ is non-vanishing. Then 
$\dot{C}(t)$ extends to a homotopically trivial vector field $v$ on 
$L$. For example, we can view $L$ as an $S^1$ bundle over 
$S^1$ with fibers representing the class $[C]\in H_1(L,\Z)$, and $C$ is one of the fibers. then take $v$ to be a non-vanishing 
vector field tangent to the fibers. 
}
\end{exam}

\begin{theo}[\cite{EP2}] 
The linking class $\ell_L=0$ for 
any embedded Lagrangian torus $L\subset \R^4$. 
\end{theo}

\begin{rem} \label{ellJ}
{\rm Given an embedded torus $L\subset \R^4$ we consider the set $\cJ^+(L)$ 
of almost complex structures $J$ defined on $T_L\R^4$ 
such that $J(TL)\pitchfork TL$ and $J$ is compatible with the orientation of $\R^4$. 
The homotopic classes of such $J$ is isomorphic to $H^1(L,\Z)\cong \Z^2$. 
Similar to $\ell_L$ for $L$ being Lagrangian, each $J$ associates a linking class 
$\ell_L(J)\in H^1(L,\Z)$ defined by linking numbers 
$\ell_L(J)(\gamma):=lk (C+\epsilon Jv, L)$, where $C,v$ are as defined above.

Then  $\ell_L(J_0)=\ell_L=0$ if $L$ is Lagrangian and $J_0$ is the standard 
complex structure (or any $\w$-compatible one.) It will be shown later that 
the vanishing of $\ell$ implies for $\phi_s$ as in 
Theorem \ref{L0L1}, $((\phi_1)_*J_0)|_{L_1}$ and $J_0|_{L_1}$ are homotopic in $\cJ^+(L_1)$. Hence 
for any embedded oriented closed curve $C\subset L_0$, 
$(\phi_1)_*N^\w_C=N^\w_{\phi_1(C)}$ 
up to a smooth isotopy rel $L_1$, 
which leads to Theorem \ref{L0L1}. 
Here $N^\w_C$ is the symplectic normal bundle as defined in 
the beginning of the next section. 
} 
\end{rem}

\section{Loops in $\R^4$ and their framings} \label{loop}

Before moving on to Lagrangian tori in $\R^4$, it helps to have a closer look at loops in $\R^4$ at first. 

 A loop in $\R^4$ is an embedded  1-dimensional submanifold diffeomorphic to $S^1$. The pull-back of $\w$ on a loop vanishes, hence a loop is an isotropic submanifold. 
Take a loop $C\subset \R^4$. We fix an orientation of $C$ and fix a trivialization 
of $C\cong S^1=\R/2\pi\Z$, and write $\dot{C}(t)$ for the tangent vector of $C$ at $C(t)$. 

\vspace{.2in} 
\noindent 
{\bf Symplectic normal bundle.} \ 
Let us recall some basic properties of the normal bundle $N$ of $C$. 
The  bundle $N$  splits as 
\[ 
N= (T^*C)\oplus N^\w, 
\] 
where $N^\w$,  called the {\em symplectic normal} bundle of $C$,  is the  
trivial $\R^2$-bundle over $C$ defined by 
\[ 
N^\w:=\{ (C(t), v)\mid t\in S^1, \  v\in N|_{C(t)}, \ \w(\dot{C}(t),v)=0\}. 
\]  
By Weinstein's isotropic neighborhood theorem (see \cite{W1,W2,MS}), there exists a tubular neighborhood $U\subset \R^4$ of $C$, a tubular neighborhood $V\subset N$ of the zero section of the normal bundle $C\subset\R^4$, and a symplectomorphism with $C\subset U$ identified with the zero 
section of $N$: 
\[ 
(U\subset \R^4,\w)\to (V\subset N=T^*C\times \R^2, \w_C\times \wc).  
\] 
Here 
$\wc=dx\wedge dy$ is the standard symplectic 2-form on $\R^2$, $\w_C=dt\wedge dt^*$ is the canonical symplectic 2-form on $T^*C$, and 
$t^*$ is the fiber coordinate of $T^*C$ dual to $t$. The symplectic normal bundle 
 $N^\w$ is identified with $\{ (t,0,x,y)\in S^1\times \R\times \R^2\}$.

Below we explore some properties of $N^\w$ that will be applied in 
later sections. 

\vspace{.2in} 
\noindent 
{\bf Lagrangian tori associated to a loop.} \ 
Let 
\[ 
D^\w\subset N^\w
\] 
denote the associated symplectic normal disc bundle with fiber an open disc $\{ (x,y)\in \R^2, \ x^2+y^2<\epsilon\}$ with some 
positive radius $\epsilon$. With the symplectomorphism near $C$ described as above, 
the boundary $L=L_C:=\partial D^\w$ is a an embedded Lagrangian torus in $\R^4$ provided that $\epsilon>0$ is small enough. Note that for each $\epsilon$ small enough, $L_C$ 
with $\epsilon>0$ fixed is unique up to 
a Hamiltonian isotopy.

It is well known that any two loops in $\R^4$ are smoothly isotopic. 
The following proposition can be easily verified. 

\begin{prop}   \label{Ds}   
Let $C_s$, $0\leq s\leq 1$, be a smooth isotopy of loops. Let $D^\w_s$ denote the symplectic 
normal disc bundle of $C_s$ with fiber radius $\epsilon_s>0$. Let $L_s:=\partial D^\w_s$. Then there exists an $\epsilon>0$ such  
that $L_s$ is a Lagrangian isotopy of embedded Lagrangian tori provided that $0<\epsilon_s<\epsilon$. In particular, if $C_0=C_1$ as a set and $\epsilon_0=\epsilon_1$, then 
 $D^\w_0=D^\w_1$ and we get a Lagrangian self-isotopy of $L_0=\partial D^\w_0$. 
\end{prop} 
In Section \ref{lmg} 
we will use this observation to construct Lagrangian self-isotopies of a Clifford torus.

\vspace{.2in} 
\noindent
{\bf Framings of $N^\w$.} \  
\begin{defn} 
{\rm 
To a non-vanishing 
section (i.e., a framing) $\sigma$ of $N^\w$ one can associate an $S^1$-family of  Lagrangian planes 
\[ 
\dot{C}(t)\wedge \sigma(t), \quad t\in S^1. 
\] 
We denote the corresponding 
Maslov index by 
\[ 
\mu_C(\sigma):=\mu(\dot{C}(t)\wedge \sigma(t))\in 2\Z. 
\] 
Note that $\mu_C(\sigma)$ depends only on the orientation of $C$ and the homotopy class of $\sigma$ among framings of $N^\w$.  
} 
\end{defn} 

If we fix a trivialization $\Phi:N^\w\to C\times \R^2=C\times \C^1$ then 
the homotopy classes of framings of $N^\w$  can be identified with $[S^1,\R^2\setminus \{ 0\}]=[S^1,S^1]=\Z$. Then for a map $\theta:S^1\to S^1$ of degree $m$, the Maslov index associated to the section $\sigma'(t):=e^{i\theta(t)}\sigma(t)$ is 
$\mu_C(\sigma')=\mu_C(\sigma)+2m$. In particular, there is a framing $\sigma^0$ of $N^\w$ such that $\mu_C(\sigma^0)=0$. 
We call $\sigma^0$  a {\em 0-framing} of $C$, it is unique up to homotopy.   Likewise, for each $m\in\Z$ there is a framing $\sigma^m$ of $N^\w$, $\sigma^m$ is unique up to homotopy, such that $\mu_C(\sigma^m)=2m$. 

\begin{defn} 
{\rm 
We call $\sigma^m$ an {\em $m$-framing} of $N^\w$ or, an 
{\em $m$-framing} of $C$.
}
\end{defn}  

Note that the homotopy classes of framings of $N^\w$ are classified by the framing number $\mu_C(\sigma)/2$. 

\begin{exam} 
{\rm Let $C\subset L$ be a simple closed curve representing the class $\gamma\in H_1(L,\Z)$ of a Lagrangian torus. Let $v$ be 
a non-vanishing  section of $N^\w_C\cap T_CL$. Then 
$v$ is a $\mu(\gamma)/2$-framing of $N^\w_C$. 
} 
\end{exam}

\begin{prop} \label{0frame}   
Let $C_s$, $s\in [0,1]$ be a smooth isotopy between loops $C_0$ and $C_1$. Write $C_s=\phi_s(C_0)$ where $\phi_s\in \Diff_0^c(\R^4)$ with $\phi_0=id$. 
Let $N^\w_s$ and $\sigma^m_s$ denote the symplectic normal bundle and the  $m$-framing of $C_s$ respectively. 
\begin{enumerate} 
\item Assume that $(\phi_1)_*N^\w_0=N^\w_1$. 
Then 
\[ 
\mu_{C_1}((\phi_1)_*\sigma^m_0)-\mu_{C_1}(\sigma^m_1)=
\mu_{C_1}((\phi_1)_*\sigma^0_0)-\mu_{C_1}(\sigma^0_1)\in 4\Z.  
\] 
\item If $\mu_{C_1}((\phi_1)_*\sigma^m_0)=\mu_{C_1}(\sigma^m_1)=2m$ then up to a perturbation of $\phi_s$ we may assume that $(\phi_s)_*N^\w_0=N^\w_s$ 
and $(\phi_s)_*\sigma^m_0=\sigma^m_s$. 
\end{enumerate} 
\end{prop} 

\begin{proof} 
(i).  First consider the case $m=0$. 
Fix a trivialization $S^1\cong \R/2\pi\Z\to C_0$ for $C_0$. This trivialization 
composed with $\phi_s$ becomes a trivialization of $C_s$. 
Applying Weinstein's isotropic neighborhood theorem we may symplectically identify a neighborhood of $C_s\in \R^4$ with a neighborhood of the zero section of the normal bundle $N_s$ of 
$C_s$. We can trivialize $N_s=S^1\times\R\times \R^2$ with coordinates $(t,t^*,x,y)$ so that 
\begin{itemize} 
\item $C_s=S^1\times\{ 0\} \times \{ 0\}$, 
\item  $N^\w_s=S^1\times \{ 0\} \times \R^2$, and 
\item $\sigma^0_s(t)=(t,0,\epsilon,0)$, for some $\epsilon >0$.  
\end{itemize} 

Then for each $s$,  the differential $(\phi_s)_*(t)$ at $C_0(t)$ with $t\in S^1\cong \R/2\pi\Z$ can be thought of as a smooth loop in $GL^+(4,\R)$: 
\[ 
(\phi_s)_*(t)\in \begin{pmatrix} 1 &  * \\ 0 & GL^+(3,\R) \end{pmatrix}, \quad 
(\phi_0)_*(t)=Id, \quad (\phi_1)_*(t)\in \begin{pmatrix} 1 & * & 0 \\
0 & c(t) & 0\\ 0 & * & GL(2,\R) \end{pmatrix} . 
\] 
Note that because $(\phi_1)_*(t)$ is an isomorphism, $c(t)\neq 0$ for $t\in S^1$. 
We view $(\phi_0)_*(t)=id$ as a constant loop in $GL^+(4,\R)$ parameterized by $t$. Then $(\phi_s)_*(t)$, $0\leq s\leq 1$, viewed as a family of parameterized loops in $GL^+(4,\R)$, is a free homotopy between $(\phi_0)_*(t)$ and 
$(\phi_1)_*(t)$. This implies that 
$(\phi_1)_*(t)$ is free homotopic to the trivial class of $\pi_1(GL^+(4,\R))\cong \pi_1(GL^+(3,\R))=\Z_2$.

Since the lower $3\times 3$ block of the matrix form of 
$(\phi_s)_*(t)$ is invertible, 
we can perturb $\phi_s$ by composing it with some suitable family of maps in $\text{Diff}^c_0(\R^4)$, each of them fixing $C_s$ pointwise and with the condition $(\phi_1)_*N^\w_0=N^\w_1$ preserved under the perturbation, so  that the perturbed $\phi_s$ satisfy  
\[ 
(\phi_s)_*(t)\in \begin{pmatrix} 1 &  0 \\ 0 & GL^+(3,\R) \end{pmatrix}, \quad with 
\quad (\phi_0)_*(t)=Id,
\] 
and either $(\phi_1)_*(t)=A(t)$ or $(\phi_1)_*(t)=A'(t)$, where 
\begin{equation} \label{AA'}
A(t)= \begin{pmatrix} 1 & 0 & 0 & 0  \\
0 & 1 & 0 & 0 \\ 0 & 0 & \cos kt & -\sin kt \\ 
0 & 0 & \sin kt & \cos kt \end{pmatrix}, \quad 
A'(t)=\begin{pmatrix} 1 & 0 & 0 & 0  \\
0 & -1 & 0 & 0 \\ 0 & 0 & \cos kt & \sin kt \\ 
0 & 0 & \sin kt & -\cos kt \end{pmatrix} 
\end{equation}  
for some $k\in\Z$. Note that $A'(t)$ is free homotopic to $A(t)$ by a $180^\circ$ 
rotation along the subspace spanned by its second and third column vectors. 
We can interchange the two cases $(\phi_1)_*(t)=A(t)$ and $(\phi_1)_*(t)=A'(t)$
by composing with $\phi_1$ such a rotation along $C_1$.

Now that $[ (\phi_1)_*(t)]=0$ in $\pi_1(GL^+(4,\R))$ implies that $k\in 2\Z$. 
Hence $\mu_{C_1}((\phi_1)_*\sigma^0_0)=2k+\mu(\sigma^0_1)=2k\in 4\Z$. 

The equality $\mu_{C_1}((\phi_1)_*\sigma^m_0)-\mu_{C_1}(\sigma^m_1)=\mu_{C_1}((\phi_1)_*\sigma^0_0)-\mu_{C_1}(\sigma^0_1)$ follows from 
the property that $\sigma^m_s(t)=e^{imt}\sigma^0_s(t)$ up to 
homotopy.

(ii). The proof follows from the perturbation of $\phi_s$ constructed in (i). 

\end{proof}

\section{Lagrangian monodromy group (LMG) of a  Clifford torus} \label{lmg} 

In general, the LMG $\cL(L)$ has to preserve both the 
Maslov class $\mu_L$ and the linking class $\ell_L$ whenever 
defined. However, for $L\subset \R^4$ the class $\ell_L=0$ is 
automatically preserved. In this section we will determine  the 
LMG of a Clifford torus in $\R^4$. 

Identify $\R^4\cong \C^2$. 
For $a,b>0$ the {\em Clifford torus} $T_{a,b}$ is defined to be 
\[ 
T=T_{a,b}:=\{ (z_1,z_2)\mid |z_1|=a, \ |z_2|=b\}. 
\] 
We fix a basis$\{ \gamma_1,\gamma_2\}$ of $H_1(T,\Z)$ so that 
\[ 
\text{ $\gamma_1$ is represented by the cycle $\{ (ae^{i t},b)\mid t\in \R/2\pi\Z\}$,} 
\] 
\[ 
\text{ $\gamma_2$ is represented by the cycle $\{ (a,be^{i t})\mid t\in \R/2\pi\Z\}$.} 
\]
Then $\gamma_1= \begin{pmatrix} 1\\0\end{pmatrix}$, $\gamma_2=\begin{pmatrix} 0\\1\end{pmatrix}$ when expressed as column vectors. We also denote $\gamma_0:=-\gamma_1+\gamma_2$. Then  $\mu(\gamma_0)=0$ and $\gamma_0=
\begin{pmatrix} -1\\1\end{pmatrix}$ as a column vector. Likewise, the Maslov class 
$\mu\in H^1(T,\Z)$ is expressed as a row vector $\mu=\begin{pmatrix}2 & 2\end{pmatrix}$. 

The mapping class group of $T$ is then isomorphic to $GL(2,\Z)$, the group of of $2\times 2$ matrices with integral coefficients and with determinant $\pm1$. 
Let 
\[ 
G_\mu:=\{ g\in GL(2,\Z)\mid \mu\circ g=\mu\}. 
\] 
A direct computation shows that $G_\mu=G^+_\mu\sqcup G^-_\mu$, where 
\begin{gather} 
G^+_\mu=\Big\{ g_n:= \begin{pmatrix} 1-n & -n\\ n & 1+n\end{pmatrix}\mid n\in \Z\Big\} \label{g+} \\ 
G^-_\mu=\Big\{ f_n:=\begin{pmatrix} 1-n & 2-n\\ n & -1+n\end{pmatrix}\mid  n\in \Z\Big\} \label{g-} 
\end{gather} 
Elements of $G^+_\mu$ are of determinant  1, and elements of $G^-_\mu$ are of determinant -1. Also $g_n=(g_1)^n$, where $g_1$ is a generator of $G^+_\mu\cong \Z$.  
 On the other hand, $G^-_\mu$ consists elements 
of order 2 in $G_\mu$. Geometrically $g_n=(g_1)^n$ is the $(-n)$-Dehn twist along $\gamma_0$ while   each of 
$f_n$ is a generalized reflection with $f_n(\gamma_0)=-\gamma_0$. Note that 
\[ 
f_0^2=e=f_1^2, \quad (f_1f_0)^n=g_n, \quad (f_0f_1)^n=g_{-n}=(g_n)^{-1},\quad g_nf_m=f_{n+m}. 
\]  
Here $e$ denotes the identity element of $G_\mu$. 
Therefore 
\begin{equation} \label{gmu} 
G_\mu=\langle f_0,  f_1 \mid f_0^2=e=f_1^2\rangle \cong D_\infty 
\end{equation} 
is freely generated by the two elements $f_0,f_1$ of order 2, and is isomorphic to the infinite dihedral group $D_\infty$ \cite{Hum}. 

Note that if $L_s=\phi_s(T)$, $s\in[0,1]$, is a Lagrangian self-isotopy of $T$ so that $L_0=L_1=T$, $\phi_0=id$, then the induced isomorphism $(\phi_1)_*:H_1(T,\Z)\to H_1(T,\Z)$ is an element of $G_\mu$. I.e., the LMG $\cL(T)$ is a subgroup of $G_\mu$.

\begin{prop}  \label{lmgab} 
The LMGs of $T_{a,b}$ and $T_{a',b'}$ are isomorphic. 
\end{prop} 

\begin{proof} 
Identify the ordered pairs $(a,b), (a',b')$ with the coordinates of two points in the first quadrant of the $\R^2$ plane. Take a smooth path $c(s)=(c_1(s),c_2(s))$, $s\in [0,1]$, in the fist quadrant so that 
$c(0)=(a,b)$, $c(1)=(a',b')$, then $T_{c(s)}$ is a Lagrangian isotopy of Clifford tori between $T_{a,b}$ and $T_{a',b'}$. 
\end{proof} 

\begin{theo} 
The LMG of a Clifford torus $T$ is  $\cL(T)=G_\mu$. 
\end{theo} 

\begin{proof} 
 We will explicitly construct Lagrangian self-isotopies of $T$ with monodromies $f_0$ and $f_1$ respectively. Then $\cL(T)=G_\mu$ 
 following (\ref{gmu}). 

\vspace{.1in} 
\noindent 
{\bf Case 1: The  monodromy $f_1=\begin{pmatrix} 0 & 1\\ 1 & 0\end{pmatrix}$}. \  Recall that in \cite{Y1} we have constructed a Lagrangian self-isotopy for $T_{b,b}$ with  monodromy $f_1$ (denoted by $\tilde{f}_1$ in \cite{Y1}). For completeness we repeat the construction here. First let us consider the 
path in the unitary group $U(2)$ defined by 
\[ 
A_s:=\begin{pmatrix} \cos \frac{\pi  s}{2} & -\sin \frac{\pi  s}{2} \\ \sin \frac{\pi  s}{2} & \cos \frac{\pi  s}{2} \end{pmatrix}\in GL(2,\C), \quad 0\leq s \leq 1. 
\] 
$A_s$ acts on $\C^2$ and is the time $s$ map of the Hamiltonian vector field 
$X=\frac{\pi}{2}(x_1\partial_{x_2}-x_2\partial_{x_1}+y_1\partial_{y_2}-y_2\partial_{y_1})$, $\w(X, \cdot )=-dH$, $H=\frac{\pi}{2}(x_2y_1-x_1y_2)$. Observe that $A_1(T_{a,b})=T_{b,a}$,  $(A_1)_*=f_1$ on $H_1(T_{b,b},\Z)$.
Fix $b>0$ and modify $H$  to get a $C^\infty$ function $\tilde{H}$ with compact support such that $\tilde{H}=H$ on $\{ |z_1|\leq 2b, \ |z_2|\leq 2b\}$. Let $\phi_s$ be the time $s$ map of the flow of the Hamiltonian vector field associated to $\tilde{H}$. Then $\phi_1(T_{b,b})=(T_{b,b})$, and $(\phi_1)_*=(A_1)_*=f_1$ on $H_1(T_{b,b},\Z)$. 
Now extend this self-isotopy of $T_{b,b}$ by conjugating it smoothly by a 
Lagrangian isotopy between $T_{a,b}$ and $T_{b,b}$ as described in Proposition \ref{lmgab}. 
We may assume that the basis $\{\gamma_1,\gamma_2\}$ of $T_{b,b}$ is transported to the basis $\{\gamma_1,\gamma_2\}$ of $T_{a,b}$ along the latter isotopy. Readers can check now that the extended isotopy induces a Lagrangian self isotopy of $T_{a,b}$ with monodromy $f_1$. 

\vspace{.1in} 
\noindent 
{\bf Case 2: The monodromy $f_0=\begin{pmatrix}  1 & 2\\ 0 & -1\end{pmatrix}$.} \  
For $s\in [0,1]$ consider the family of diffeomorphisms $\Psi_s:\R^4\to \R^4$, 
\[ 
\Psi_s(x_1,y_1,x_2,y_2):=(x_1\cos \pi s -y_2\sin \pi s, y_1,x_2, y_2\cos \pi s+x_1\sin \pi s). 
\] 
Note that $\Psi_s\in SO(4,\R)$  are  rotations on the $x_1y_2$-plane, with the 
$y_1x_2$-plane fixed.   
Consider the simple closed curve $C_0$ defined by 
\[ 
\{ (x_1=0, y_1=0, x_2=b\cos t, y_2=b\sin t)\in \R^4 \mid  t\in [0,2\pi]\}. 
\] 
Define $C_s(t):=\Psi_s(C_0)(t)$. $C_s$, $s\in [0,1]$ is a smooth family of curves. Note that $C_1$ equals $C_0$ but with the reversed orientation. Recall from Proposition \ref{Ds}  that for $\epsilon>0$ small enough, the Lagrangian torus boundary $L_s$ of the symplectic normal disc bundle $D^\w_s$ of radius $\epsilon$ of $C_s$ are embedded in $\R^4$, with core curve $C_s$. Note that $L_0=T_{\epsilon,b}=L_1$ as sets, so we obtain a Lagrangian self-isotopy of $T_{\epsilon,b}$ for $\epsilon>0$ small enough. 
This  self-isotopy of $T_{\epsilon,b}$ reverses the orientation of $T_{\epsilon,b}$, so the corresponding 
monodromy $f$ is an element of $G^-_\mu$, with determinant -1 when expressed as a matrix. 
Note that $\Psi_1$ reverses the orientation of the core curve $C_0$ 
of $D^\w_0$. Since $\gamma_2\subset \partial D^\w_0=T_{\epsilon,b}$ is longitudinal, this implies that $f$ sends $\gamma_2$ to $-\gamma_2+m\gamma_1$ for 
some $m\in\Z$. Then by comparing with the formula of $f_n$ in (\ref{g-}) one finds that $f=f_0=\begin{pmatrix}  1 & 2\\ 0 & -1\end{pmatrix}$ and $m=2$. 

Now, similar to what is done in Case 1, extend the Lagrangian self-isotopy of $T_{\epsilon,b}$ into an Lagrangian self-isotopy of $T_{a,b}$ through Clifford tori. The corresponding monodromy is $f_0$. 
This completes the proof. 
\end{proof} 

\begin{rem} 
{\rm 
If we take $C_0$ to be the curve  
\[ 
\{ (x_1=a\cos t, y_1=a\sin t,x_2=0,y_2=0) \in \R^4 \mid t\in [0,2\pi]\},  
\] 
then $\Psi_s$ will induce a Lagrangian self-isotopy of $T_{a,\epsilon}$ with monodromy 
$f_2=\begin{pmatrix}  -1 & 0\\ 2 & 1\end{pmatrix}$. 
The reader can check that $G_\mu=\langle f_1,f_2\mid f_1^2=e=f_2^2\rangle$, Hence $\cL(T)=G_\mu$ again. 
} 
\end{rem}

\section{Smooth Monodromy Group (SMG) of a Clifford torus} \label{smg}

We start with the proof of Theorem \ref{L0L1}. 

\begin{proof}[{\bf Proof of Theorem \ref{L0L1}}.] \  By the linearity of 
$(\phi_1)_*$ and $\mu$ it is enough to prove for the case when 
$\gamma\in H_1(L_0,\Z)$ is primitive. 

Fix a positive basis $\{ \gamma_1,\gamma_2\}$ for $H_1(L_1,\Z)$ with $\mu(\gamma_1)=2=\mu(\gamma_2)$. Given a primitive class $\gamma\in 
H_1(L_0,\Z)$ we have 
$(\phi_1)_*(\gamma)=n_1\gamma_1+n_2\gamma_2$ for some 
$n_1,n_2\in \Z$. 
Let $C_0\subset L_0$ be an embedded curve representing the class $\gamma$. 
Let $C_s:=\phi_s(C_0)$. We denote by $N_s$ and $N^\w_s$ respectively the normal bundle and the symplectic normal bundle of $C_s$. 
By assumption $C_1$ represents the class $n_1\gamma_1+n_2\gamma_2$. 

Let $\s_0$ denote a non-vanishing  section  of the 
$\R^1$-bundle $(T_{C_0}L_0)\cap N^\w_0$ over $C_0$. 
Then 
$\s_0$ is a $\mu(\gamma)/2$-framing of $N^\w_0$. 
Extend $\s_0$ to a smooth family 
$\s_s$ with $0\leq t\leq 1$, so that $\s_s$ is a $\mu(\gamma)/2$-framing of $N^\w_s$. Let $m:=\mu(\gamma)/2$.

Recall $J_0$ is the standard complex structure over $\R^4\cong \C^2$.
Fix a trivialization for $N_s\cong S^1\times \R\times \R\times \R$ 
by taking $\{ J_0\dot{C}_s(t),\s_s(t),J_0\s_s(t)\}$ as the basis of 
the fiber of $N_s$ at $C_s(t)$,  so that 
the coordinate $(t,t^*, x,y)$ represents the fiber $t^*J_0\dot{C}_s(t)+x\s_s(t)+
yJ_0\s_s(t)$. 

Now let $\eta_s:=\phi_s(\sigma_0)$. Note that $\eta_1$ is a non-vanishing 
section of $N^\w_1\cap T_{C_1}L_1$, 
and an $(n_1+n_2)$-framing of $N^\w_1$. 
Let $k:=n_1+n_2$. 

Recall that $\s_1$ is an $m$-framing of $N^\w_1$. 
Up to a homotopy of $\s_s$ if necessary, we may assume the following:  
\begin{itemize} 
\item For each $s$, $\eta_s=\sigma_s$ at $t=0$. 
\item For $t\in S^1=\R/2\pi\Z$, $\eta_1(t)=\sigma_1(t)\cos  (k-m) t +J_0\sigma_1(t)\sin  (k-m)t$. 
\end{itemize} 
Then for each $s$, $\phi_s$ associates a smooth map 
$\Phi_s:S^1\to GL^+(4,\R)$, 
\[ 
\Phi_s(t):=(\phi_s)_*(t)\in \begin{pmatrix}    1 & * \\ 0 & GL^+(3,\R) \end{pmatrix}, 
\] 
\[ 
\Phi_0(t)=Id, \quad \Phi_1(t)=\begin{pmatrix} 1 & * & 0 & * \\ 
0 & * & 0 & * \\ 
0 & * & \cos (k-m)t & * \\ 0 & * & \sin (k-m)t & * 
\end{pmatrix} . 
\] 
Then second and fourth columns of $\Phi_1$ represent $(\phi_1)_*(J_0\dot{C}_0)$ and $(\phi_1)_*(J_0\sigma_0)$ respectively.

Extend $\dot{C}_0$ to a homotopically trivial non-vanishing vector field $u_0$ on $L_0$. 
Let $u_s:=(\phi_s)_*u_0$. Then $u_1|_{C_1}=\dot{C}_1$. 
By continuity and $\ell_{L_0}=0$ we have 
\begin{equation} \label{lk1} 
lk(C_1+\epsilon\cdot 
(\phi_1)_*J_0u_0,L_1)=lk(C_0+\epsilon J_0u_0,L_0)=0. 
\end{equation} 
 Similarly, since $\ell_{L_1}=0$, 
\begin{equation} \label{lk2} 
lk(C_1+\epsilon J_0(\phi_1)_*u_0,L_1)=
lk(C_1+\epsilon J_0u_1,L_1)=0. 
\end{equation} 
Note that (\ref{lk1}) and (\ref{lk2}) holds true for any 
class $[C_0]$ and hence $[C_1]=(\phi_1)_*[C_0]$. 
This shows that $(\phi_1)_*J_0|_{L_1}$ is homotopic to 
$J_0|_{L_1}$ in $\cJ^+(L_1)$ as defined in Remark \ref{ellJ}. 
In particular, 
$(\phi_1)_*J_0u_0$ is homotopic to $J_0u_1$ as non-vanishing 
sections of the normal bundle $N_{L_1}$ of $L_1\subset \R^4$. 
So up to an $L_1$-fixing isotopy we may assume that 
along $C_1$, 
$(\phi_1)_*J_0\dot{C}_0=J_0\dot{C}_1$ and $(\phi_1)_*N^\w_{C_0}=N^\w_{C_1}$. I.e., 
$\Phi_1=(\phi_1)_*$ satisfies 
\begin{equation}\label{Phi} 
\Phi_1(t)=
\begin{pmatrix} 1 & 0 & 0 & 0 \\ 0 & 1 & 0 & 0 \\ 
0 & 0 & \cos (k-m)t & * \\ 0 & 0 & \sin (k-m)t & * \end{pmatrix} \in GL^+(4,\R). 
\end{equation} 
Now 
$\Phi_1$ satisfies the hypothesis of Proposition \ref{0frame}(i), 
so by a similar argument as employed there we have, up to 
an $L_1$-fixing isotopy, 
\[ 
\Phi_1(t)=\begin{pmatrix} 1 & 0 & 0 & 0 \\ 0 & 1 & 0 & 0 \\ 
0 & 0 & \cos (k-m)t & -\sin (k-m)t \\ 0 & 0 & \sin (k-m)t & \cos (k-m)t 
\end{pmatrix} \in GL^+(4,\R)
\] 
with 
\begin{equation}\label{k-m} 
 k-m\in2\Z
\end{equation} 
since the lower $3\times 3$ block of $\Phi_1$ is free homotopic to $Id\in GL^+(3,\R)$ with 
respect to the basis $\{ J_0\dot{C_1}, \sigma_1,J_0\sigma_1\}$. 
This completes the proof. 
\end{proof}

\begin{cor} 
The SMG $\cS(L)$ of an embedded Lagrangian torus $L\subset 
\R^4$ is contained in the subgroup  $\cX\subset \text{Isom}(H^1(L,\Z))$ 
defined by 
\[ 
\cX:=\{ g\in {Isom}(H_1(L,\Z))\mid \mu_L \circ g -\mu_L \in 4\cdot H^1(L,\Z)\} . 
\] 
\end{cor}

\begin{cor} 
Let $L\subset\R^4$ be an embedded Lagrangian torus. 
Fix a positive basis $\{ \gamma_1,\gamma_2\}$ for $H_1(L,\Z)$ with $\mu(\gamma_1)=2=\mu(\gamma_2)$. Then with 
respect to $\{  \gamma_1,\gamma_2\}$, $\cX$ is represented as 
\begin{align} 
\cX & =\cX^o\sqcup \cX^e\subset GL(2,\Z),  \nonumber \\ 
\cX^o & :=\Big\{ \begin{pmatrix} 1+2p & 2s\\ 2r & 1+2q\end{pmatrix}\in GL(2,\Z)\mid p,q,r,s\in \Z\Big\},  \\ 
\cX^e & :=\Big\{ \begin{pmatrix} 2r & 1+2q \\ 1+2p & 2s\end{pmatrix}\in GL(2,\Z)\mid p,q,r,s\in \Z \Big \}.  
\end{align} 
\end{cor} 

\begin{proof} 
Recall that $\mu=\mu_L$ has divisibility 2. Express $\gamma_1,
\gamma_2$ as column vectors $\begin{pmatrix} 1 \\ 0\end{pmatrix}$ and $\begin{pmatrix} 0 \\ 1\end{pmatrix}$ respectively. 
For $g=(g_{ij})\in \cX$, that $\mu(g(\gamma_j))-\mu(\gamma_j)
\in 4\Z$ implies that both $2(g_{11}+g_{21})-2$ and $2(g_{12}+g_{22})-2$ are divisible by $4$. Hence 
(i) $g_{11}$ and $g_{21}$ have different parity, and (ii) $g_{12}$ and $g_{22}$ have different parity. 
Since $\det g=\pm 1$,  the two even valued entries of $g$ cannot lie in the same column nor the same row of $g$, hence either $g\in \cX^o$ or $g\in \cX^e$. 
\end{proof}

We now move on to determine the group $\cS(T)$ of a Clifford torus $T$.  The proof is divided into the following three propositions.

\begin{prop} \label{mcg} 
Recall the basis $\{ \gamma_1,\gamma_2\}$ for $H_1(T_{a,b},\Z)$. 
Each of the following four types of elements of $GL(2,\Z)\cong \text{Isom}(H_1(T_{a,b},\Z))$ can be realized as the monodromy of some smooth self isotopy of $T_{a,b}$: 
\begin{enumerate} 
\item a $k$-Dehn twist $\tau_1^k:=\begin{pmatrix}1 & k\\ 0&1\end{pmatrix}$ along $\gamma_1$ with $k\in2\Z\setminus \{ 0\}$, 
\item a $k$-Dehn twist $\tau_2^k:=\begin{pmatrix}1 & 0\\ -k&1\end{pmatrix}$ along $\gamma_1$ with $k\in2\Z\setminus \{ 0\}$, 
\item the $\gamma_1$-reflection  $\bar{r}_1:=\begin{pmatrix}-1 & 0\\ 0&1\end{pmatrix}$, 
\item  the $\gamma_2$-reflection  $\bar{r}_2:=\begin{pmatrix}1 & 0\\ 0& -1\end{pmatrix}$. 
\end{enumerate} 
\end{prop}

\begin{proof} 
Since the specific values of $a,b>0$ are immaterial, we may take values of $a,b$ that are convenient 
for the construction of a smooth self-isotopy. In the 
following we will denote a Clifford torus as $T$. 
Also, since the Lagrangian monodromy $f_1=\begin{pmatrix} 0 & 1\\ 1 & 0\end{pmatrix}$ swaps elements in (i)(iii) with elements in (ii)(iv), we only have to 
prove the two cases: (i) and (iii). 

Let $C:=\{ (0,be^{it}\mid t\in[0,2\pi]\}\subset \R^4$. 

\vspace{.1in} 
\noindent
{\bf Case (i): $\tau_1^k$, $k\neq 0$ is even.} \  

Let $U$ be a tubular neighborhood of $C$, $U\cong B^3\times S^1$. Parameterize $U$ 
by $(\rho, \varphi,\theta, t)$ where $(\rho,\varphi,\theta)\in [0,\rho_0]\times S^2$ are the spherical coordinates of the 3-ball $B^3$ with $\rho$ being the radial coordinate, $(\varphi, \theta)$ being spherical coordinates on $S^2$ and $(\rho_0,\pi/2,\theta,t)$ parameterizes the equator of the $S^2$-fiber over $t$. We also assume that $(\rho_0,\pi/2,\theta, t)\in S^1\times S^1$ parameterize $T$ so 
that $\tau_1^k$ is represented by the map $\phi(\theta,t)=(\theta+kt, t)$.  Extend $\phi$ over $U$ to get 
\[
\tilde{\phi}:U\to U, \quad \tilde{\phi}(\rho, \varphi,\theta,t)=(\rho,\psi_t(\varphi,\theta),t):=(\rho,(\varphi,\theta+kt),t). 
\] 
As a loop in $SO(3)$ parameterized by $t$, the maps $\psi_t$ represents the trivial class of $\pi_1(SO(3))$, following the assumption that $k$ is even. Then there exists between $\psi_t$ and the constant loop $Id$ a smooth homotopy $\psi_{s,t}\in SO(3)$, $s,t\in [0,1]\times S^1$, such that $\psi_{0,t}=Id=\psi_{s,0}$, $\psi_{1,t}=\psi_t$. This induces a smooth homotopy $\tilde{\phi}_s$, $s\in [0,1]$, between $\tilde{\phi}_1=\tilde{\phi}$ and $\tilde{\phi}_0=id_U$ with  
\[ 
\tilde{\phi}_s(\rho,(\varphi,\theta),t):=(\rho, \psi_{s,t}(\varphi,\theta),t). 
\] 
Let $X_s$ be the time dependent vector field on $U$ that generates the isotopy $\tilde{\phi}_s$, i.e., $\dfrac{d\tilde{\phi}_s}{ds}=X_s\circ\tilde{\phi}_s$, $\tilde{\phi}_0=id$. 
Note that $X_s$ is tangent to $\partial U$. Extend $X_s$ over $\R^4$ smoothly with compact support. Denote the time 1 map of the extended $X_s$ as $\phi'$. Then $\phi'\in \Diff_0^c(\R^4)$  is isotopic to the identity map, and $\phi'|_L=\phi$. 

\vspace{.1in} 
\noindent
{\bf Case (iii):  $\bar{r}_1$.} \ 

Parameterize $B^3$ by Cartesian coordinates $(x_1,y_1,x_2)$ with $x_1^2+y_1^2+x_2^2\leq 1$ so that $T\subset U=B^3\times S^1$ is parameterized by 
$\{ (x_1,y_1,0,t)\mid x_1^2+y_1^2=1\}$. 
Without loss of generality we may assume that $\bar{r}_1$ is represented by the map 
$\phi(x_1,y_1,0,t)=(-x_1,y_1,0,t)$ for $(x_1,y_1,0,t)\in T$. 
Extend $\phi$ over $U$ to get 
\[ 
\tilde{\phi}:U\to U, \quad \tilde{\phi}(x_1,y_1,x_2,t)=(\psi(x_1,y_1,x_2),t):=((-x_1,y_1,-x_2),t). 
\] 
The map $\psi=\begin{pmatrix} -1& 0 & 0\\ 
0 & 1 & 0 \\ 0 & 0 & -1 \end{pmatrix} \in SO(3)$ is isotopic to the identity map. 
Let $\psi_s$ be a smooth path in $SO(3)$ with $s\in[0,1]$, $\psi_0=Id$ and $\psi_1=\psi$. 
This induces an isotopy $\tilde{\phi}_s:U\to U$, $s\in [0,1]$,  
\[ 
\tilde{\phi}_s((x_1,y_1,x_2),t)=(\psi_s(x_1,y_1,x_2),t). 
\] 
Now we extend $\tilde{\phi}_s$ over $\R^4$ with compact support just as in (i) to 
get $\phi'\in\Diff_0^c(\R^4)$ which is isotopic to the identity map, and 
$\phi'|_L=\phi$. 
This completes the proof. 
\end{proof}

Let 
\[ 
\cR\subset GL(2,\Z) 
\] 
 be the subgroup generated by elements of $\cL(T)=G_\mu$ and by $\tau_j^2$, $\bar{r}_j$ for $j=1,2$. Clearly we have the following 
inclusions as subgroups: 
\[ 
\cR\subset \cS(T)\subset \cX. 
\] 
Below we will show that $\cX\subset \cR$, hence $\cR=\cS(T)=\cX$. To begin with, 
let us consider  the subgroup 
$\cE\subset GL(2,\Z)$  generated by $\tau_1^2$ and 
$\tau_2^2$. It is shown by Sanov \cite{San} that $\cE$ is free (see also \cite{Br}) and 
\[ 
\cE=\Big\{ \begin{pmatrix} 1+4p & 2s\\ 2r & 1+4q \end{pmatrix}\in GL(2,\Z)\mid 
p,q,r,s\in \Z\Big\}. 
\]

\begin{prop} \label{rsx} 
The group $\cX$ is contained in $\cR$. Hence $\cR=\cS(T)=\cX$. 
\end{prop} 

\begin{proof} 
Since $\cX^e=f_1\cX^o$ and $f_1\in \cR$, it suffices to show that if $h\in \cX^o$ 
then $h\in\cR$. Our strategy here is to show that for $h\in \cX^o$ there exists a 
suitable element $g\in \cR$ such that $gh\in \cE$. Then $h=g^{-1}(gh)\in \cR$. 

Write $h=\begin{pmatrix} 1+2p & 2s \\ 2r & 1+2q \end{pmatrix}$. We divide the proof into four cases according to the parity of $p$ and $q$: 

\begin{enumerate} 
\item If both $p$ and $q$ are even, then already $h\in \cE\subset \cR$. 

\item If both $p$ and $q$ are odd, then 
\[  
(\bar{r}_1\bar{r}_2)h=\begin{pmatrix} -1 & 0\\ 0 & -1\end{pmatrix} 
\begin{pmatrix} 1+2p & 2s \\ 2r & 1+2q \end{pmatrix}=
\begin{pmatrix} 1-2(1+p) & -2s \\ -2r & 1-2(1+q)\end{pmatrix}  \in \cE. 
\] 
Hence $h\in \cR$ because $\bar{r}_1,\bar{r}_2\in\cR$. 

\item If $p$ is odd and $q$ is even, then again 
\[ 
\bar{r}_1h=\begin{pmatrix} -1 & 0\\ 0 & 1\end{pmatrix} 
\begin{pmatrix} 1+2p & 2s \\ 2r & 1+2q \end{pmatrix}=
\begin{pmatrix} 1-2(1+p) & 2s \\ -2r & 1+2q\end{pmatrix} \in \cE,  
\] 
and we have $h\in \cR$. 

\item The case that $p$ is even and $q$ is odd is similar: simply observe that $\bar{r}_2h\in \cE$. 
\end{enumerate} 

Thus we have proved that $\cX\subset \cR$ and hence 
$\cS(T)=\cX=\cR$. 
\end{proof}

\begin{prop} \label{cX} 
The group $\cS(T)\subset GL(2,\Z)$ is generated by $f_1$, $f_2$ and $\bar{r}_1$. 
\end{prop} 

\begin{proof} 
Recall that $\cS(T)=\cR$ is generated by $\bar{r}_j$ and $\tau_j^2$ with $j=1,2$,  and by elements of $G_\mu$. The group $G_\mu$ is generated by $f_1$ and $f_0$. 
Observe that 
\[ 
\tau_1^2=\bar{r}_2f_0, \quad \tau_2^2=f_2\bar{r}_1=f_1f_0f_1\bar{r}_1, 
\quad  \bar{r}_2=f_1\bar{r}_1f_1, . 
\] 
So indeed $\cS(T)$ is generated by the three elements $f_0,f_1,\bar{r}_1$ of order 2. Note that $(\bar{r}_1f_1)^{-1}=f_1\bar{r}_1=-\bar{r}_1f_1$, $(\bar{r}_1f_1)^2=(f_1\bar{r}_1)^2=-e$.  The element  $-e$ commutes with every element of $\cS(T)$. 
\end{proof} 

This concludes the proof of Theorem \ref{cS}.

\section{Proof of Proposition \ref{StoL}} \label{stol}

We divide the proof into two steps. In Step 1 we show that there 
exists a smooth isotopy $\phi_s$ with  $\phi_1(T)=L$ such that 
$\phi_1^*\mu_L=\mu_T$. In Step 2 we modify $\phi_s$ so that 
$\phi_s(T\setminus D)$ is Lagrangian for all $t$.  

\vspace{.2in} 
\noindent
{\bf Step 1:} \  
Let $\psi_s\in \Diff_0^c(\R^4)$, $s\in [0,1]$, be a smooth isotopy with $\psi_0=id$ and $\psi_1(L)=T$. Then 
$\psi_1^*\mu_L-\mu_T\in 4\cdot H^1(T,\Z)$ by Theorem 
\ref{L0L1}, and hence $\psi_1^*\mu_L=\mu_T\circ g$ for some 
$g\in \cX_T$. Since $\cX_T=\cS(T)$ 
by Proposition \ref{rsx}, there exists a smooth self-isotopy $\psi'_s$ of $T$ with $(\psi'_1)_*=g^{-1}$ and hence 
$(\psi'_1)^*(\psi_1^*\mu_L)=(\psi'_1)^*(\mu_T\circ g)=\mu_T$. 

Now define 
\[ 
\phi_s=\begin{cases} 
\psi'_{2s} & \text{ for $0\leq s\leq 1/2$}, \\ 
\psi_{2s-1}\circ \psi'_1 & \text{ for $1/2\leq s\leq 1$}. 
\end{cases} 
\] 
Then $\phi_s\in \Diff_0^c(\R^4)$, $\phi_0=id$, $\phi_1(T)=L$, and 
$\phi_1^*\mu_L=(\psi_1\circ\psi'_1)^*\mu_L=(\psi'_1)^*\psi_1^*\mu_L=\mu_T$. 

Let $L_s:=\phi_s(T)$ for $s\in[0,1]$. Then $L_0=T$ and $L_1=L$. 

\vspace{.2in} 
\noindent 
{\bf Step 2:} \  
We can improve the smooth isotopy $L_s$ so that it is 
indeed a Lagrangian isotopy outside a disc:

\begin{lem} \label{x} 
Let $L_s=\phi_s(L_0)$, $s\in[0,1]$, be a smooth isotopy between a Clifford torus $T=L_0$ and a Lagrangian torus $L=L_1$ with $\phi_s\in \Diff_0^c(\R^4)$, 
$\phi_0=id$, and $\phi_1^*\mu_L=\mu_T$. Then there exists an smooth isotopy 
$L'_s=\phi'_s(L'_0)$ between $T=L'_0$ and $L=L'_1$ and a disc $D\subset T$ 
such that $L'_s\setminus \phi'_s(D)$ is Lagrangian for all $s\in [0,1]$. 
\end{lem}

\begin{proof} 
Take two simple curves $\gamma,\gamma'\subset T$ which generate $H_1(T,\Z)$,  and $\gamma$ intersects with $\gamma'$ exactly at one point $p\in T$. Fix an orientation of $T$. 
We orient $\gamma$ and $\gamma'$ so that the homological intersection $\gamma\cdot\gamma'$ is $1$. 
Denote $\gamma_s:=\phi_s(\gamma)$ and $\gamma'_s:=\phi_s(\gamma')$ with 
induced orientations. Also let $p_s:=\phi_s(p)$. 

Let us start with $\gamma_s$. Let $2m=\mu_T(\gamma_0)=\mu_L(\gamma_1)$. Let 
$\sigma^m_s\subset N^\w_s$ denote the $m$-framing of the symplectic normal bundle $N^\w_s$ of 
$\gamma_s$, so $\mu_{\gamma_s}(\sigma^m_s)=2m$. Clearly we may take $\sigma^m_0$ to be a non-vanishing section 
of the normal bundle $N_{\gamma/T}$ of $\gamma=\gamma_0\subset T$.  Likewise we may take $\sigma^m_1=(\phi_1)_*(\sigma^m_0)$ since $\phi_1^*\mu_L=\mu_T$.

Trivialize the normal bundle $N_s$ of $\gamma_s$ as $N_s=S^1\times \R\times \R^2$ with coordinates 
$(t,t^*,x,y)$ so that (i) $\gamma_s=S^1\times \{ 0\}\times \{ 0\}$, (ii) $N^\w_s=S^1\times \{ 0\}\times \R^2$, and 
(iii) $\sigma^m_s(t)=(t,0,\epsilon, 0)$ for some $\epsilon>0$. This is exactly the same setup used in the proof of 
Proposition \ref{0frame}(i) except that $\sigma^0_s$ is replaced by $\sigma^m_s$ here. With respect to the trivialization of $N_s$  the differential of $\phi_s$ along $\gamma_s$ defines a loop with base point $Id$ in the subgroup 
$A\subset GL^+(4,\R)$ 
consisting of matrices of the form $\begin{pmatrix} 1 & * \\ 0 & GL^+(3,\R)\end{pmatrix}$. Note that $\phi_0$ and $\phi_1$ correspond to the constant loop. Thus 
the total of the family $\phi_s$ corresponds to a smooth map  $\Phi :I^2/\partial I \cong S^2\to A$, with $I^2=[0,1]_s\times [0,2\pi]_t$, $\Phi(s,t):=(\phi_s)_*(t)$. 
Since $\pi_2(A,Id)\cong \pi_2(SO(3,\R), Id)=0$ there exists a smooth homotopy 
$\Xi:( I^2/\partial I^2)\times [0,1]\to A$ such that $\Xi(\cdot,0 )=\Phi$, 
$\Xi(\cdot,1 )=Id$, and 
$\Xi(p,u)=Id$ for $p\in \partial I^2$, $\forall u\in [0,1]$.

This implies that for each $s$ there is a tubular neighborhood $U_s\subset\R^4$ of 
$\gamma_s$, a smooth 
family of maps $\phi_{s,u}\in \text{Diff}^+_0(\R^4)$ with $\phi_{s,0}=\phi_s$, 
$\phi_{s,u}=\phi_s$ on $\gamma_s$ and $\R^4\setminus U_s$, and $\phi_{i,u}=\phi_i$ for $i=0,1$, such that $\phi_{s,1}(T)$ is Lagrangian along 
$\gamma_s$, i.e., 
$T_{\gamma_s}\phi_{s,1}(T)$ is Lagrangian. By a further perturbation if necessary, we may assume that there exists a tubular neighborhood 
$V\subset T$ of $\gamma_0$ 
such that $\phi_{s,1}(V)$ is Lagrangian. 

Now apply the same argument to $\gamma'_s$ and $\phi_{s,1}$, just like what we have 
done for $\gamma_s$ and $\phi_s$. We then get an open neighborhood $Q\subset 
T$ of $\gamma\cup \gamma'$ with $D:=T\setminus Q$ diffeomorphic to a 2-disc, and a new isotopy $L'_s=\phi'_s(T)$ of $T=L_0$ and 
$L=L_1$ with $\phi'_s\in \Diff_0^c(\R^4)$, $\phi'_0=id$, 
such that $Q_s:=\phi'_s(Q)\subset L'_s$ is Lagrangian for $s\in [0,1]$. 
We may assume that $C_s:=\partial Q_s$ are smooth for all $s$. Take $D=T\setminus Q$. 
\end{proof} 

This completes the proof of Proposition \ref{StoL}.

\vspace{.2in } 
\noindent
Department of Mathematics \\ 
National Central University  \\ 
Chung-Li, Taiwan \\ 

\noindent 
yau@math.ncu.edu.tw


\end{document}